\documentclass{amsart}
\usepackage[utf8]{inputenc}
\usepackage{amsfonts,amsmath,amssymb,amsthm}
\usepackage{enumerate,mathrsfs,mathtools,graphicx}
\usepackage{soul,tikz-cd}
\usepackage[
left=1.5in,
right=1.5in,
bottom=1.2in
]{geometry}
\usepackage{listings}
\usepackage{mdframed,xcolor}
\definecolor{light-gray}{gray}{0.95} 

\definecolor{cadmiumgreen}{rgb}{0.0, 0.42, 0.24}
\usepackage[
colorlinks, citecolor=cadmiumgreen,
pagebackref,
pdfauthor={Rohan Pandey, Harry Richman}, 
pdfstartview ={FitV},
]{hyperref}
\hypersetup{
pdftitle={The Mobius function of the poset of triangular numbers under divisibility},
}

\usepackage[
backrefs,
msc-links,
nobysame,
lite,
]{amsrefs}

\newtheorem{prop}{Proposition}

\newtheorem{conj}{Conjecture}
\newtheorem*{conj*}{Conjecture}

\theoremstyle{definition}

\newcommand{\NN}{\mathbb{N}}
\newcommand{\ZZ}{\mathbb{Z}}

\newcommand{\Mobius}{M\"obius}
\newcommand{\leqtri}{\leq_{\tri}}
\newcommand{\mutri}{\mu_{\tri}}
\newcommand{\tri}{\mathcal T}
\newcommand{\mertens}{\mathcal M}
\newcommand{\leqP}{\preccurlyeq_P}

\title[M\"{o}bius function of triangular numbers]{The M\"obius function of the poset of triangular numbers under divisibility}
\author{Rohan Pandey}
\author{Harry Richman}
\date{\today}

\begin{document}


\maketitle

\begin{abstract}
    This paper analyzes the \Mobius{} ($\mu(i)$) function defined on the partially ordered set of triangular numbers ($\tri(i)$) under the divisibility relation. We make conjectures on the asymptotic behavior of the classical \Mobius{} and Mertens functions on the basis of experimental data and other proven conjectures. We first introduce the growth of partial sums of $\mutri(i)$ and analyze how the growth is different from the classical \Mobius{} function, and then analyze the relation between the partial sums of $|\mutri(i)|$, and how it is similar to the asymptotic classical \Mobius{} function. Which also happens to involve the Riemann zeta function. Then we create Hasse diagrams of the poset, this helps introduce a method to visualize the divisibility relation of triangular numbers. This also serves as a basis for the zeta and \Mobius{} matrices. Looking specifically into the poset defined by $(\NN, \leqtri)$, or triangular numbers under divisibility and applying the \Mobius{} function to it, we are able create our desired matrices. And then using Python libraries we create visualizations for further analysis, and are able to project previously mentioned patterns. Through which we are able to introduce two more novel conjectures bounding $\mutri(n)$ and the sums of $\frac{\mutri(i)}{i}$. We conclude the paper with divisibility patterns in Appendix~\ref{sec:triangular-patterns}, with proofs of the helpful and necessary propositions. 
\end{abstract}

\tableofcontents

\section{Introduction}

The classical \Mobius{} function is an important and well-studied function in multiplicative number theory.
It has connections to many theorems and unresolved conjectures.
Notably, the Riemann hypothesis is equivalent to a certain asymptotic bound on the partial sums of the \Mobius{} function.

In this paper, we introduce and study a variant on the classical \Mobius{} function.
The classical \Mobius{} function arises from the positive integers under division, which forms the structure of a partial order.
We consider the partial order on the positive integers which records divisibility among the triangular numbers.
In more detail: let $\tri(i) = \frac12 i(i+1)$ denote the $i$-th triangular number,
and let $\leqtri$ be the relation on $\NN$ defined by
\[
    i \leqtri j 
    \qquad\Leftrightarrow\qquad
    \tri(i) \text{ divides } \tri(j) .
\]
Equivalently, we have
$i \leqtri j$ if and only if $i(i+1)$ divides $j(j+1)$.
 
The partially ordered set $(\NN, \leqtri)$ has a \Mobius{} function, which we denote $\mutri$.
Based on experimental data, we make the follow conjectures.

\begin{conj}[Growth of partial sums of $\mutri$]
\label{conj:mobius-sum}
There is a positive constant $C$ such that
\begin{equation*}
    \sum_{i=1}^n \mutri(i) \leq - C n \quad\text{for all sufficiently large } n.
\end{equation*}
\end{conj}

\begin{conj}[Partial sums of $|\mutri|$]
\label{conj:mobius-abs-sum}
As $n \to \infty$,
\begin{equation*}
    \sum_{i=1}^n |\mutri(i)| = \frac12 n + o(n).
\end{equation*}
\end{conj}
The little-$o$ asymptotic notation here means that $\lim_{n \to \infty} \frac{1}{n} \sum_{i=1}^n |\mutri(i)| = \frac12$.

Conjecture~\ref{conj:mobius-sum} is notably different from the corresponding behavior of the classical \Mobius{} function $\mu$, for which $\sum_{i=1}^n \mu(i)$ changes sign infinitely often,
and it is conjectured that $\sum_{i=1}^n \mu(i) = O(n^{1/2 + \epsilon})$ for any $\epsilon > 0$.
Regarding Conjecture~\ref{conj:mobius-abs-sum}, the corresponding asymptotic for the classical \Mobius{} function $\mu$ is
\[
    \sum_{i=1}^n |\mu(i)| = \frac{1}{\zeta(2)}n + o(n) ,
\]
where the leading constant is $1/\zeta(2) = {6}/{\pi^2} \approx 0.608$.

\begin{figure}[h]
    \centering
    \includegraphics[width = 0.75\textwidth]{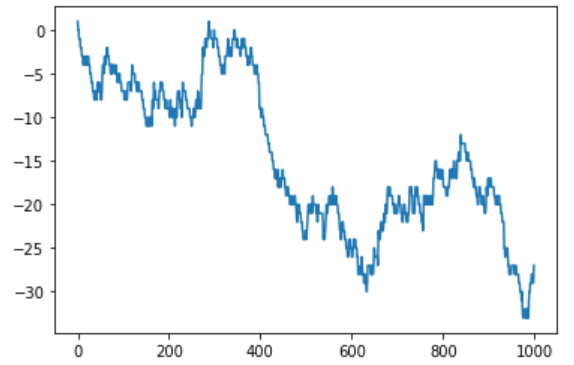}
    \caption{Partial sums of $\mutri$ from 1 to 1000.}
    \label{fig:mutri-partial-sums}
\end{figure}

\begin{conj}
    For any integer $M \geq 0$, there exists $n$ such that $\mutri(n) \geq M$.
\end{conj}

These conjectures are made on the basis of experimental data. We write code in Python to compute the \Mobius{} function $\mutri$ of the poset $(\NN, \leqtri)$, and examine plots of the relevant \Mobius{} function sums. 
The Python code is included in an appendix.
The data on the \Mobius{} function $\mutri$ and its partial sums were submitted to OEIS as entries A350682 and A351167~\cite{OEIS-mobius, OEIS-mobius-sums}.

\begin{figure}[h] 
    \centering
    \includegraphics{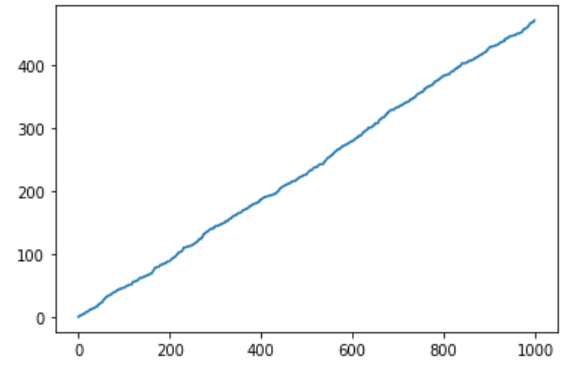}
    \caption{Partial sums of $|\mutri |$ from 1 to 1000.}
    \label{fig:mutri-abs-partial-sums}
\end{figure}

\subsection*{Acknowledgements}

Figures were created in Python~\cite{python} using Matplotlib~\cite{matplotlib} and Plotly~\cite{plotly}.

\section{Background}

Further information is included in this section to clarify terms and diagrams and other information. 
For further background on number theory, see Burton~\cite{Burton}.
For further background on the combinatorics of posets, see~\cite{Bona}.

\subsection{The classical \Mobius{} and Mertens functions}


The classical \Mobius{} function \cite{mobius}is defined, for a positive integer $n$, by the following rules:
\[
\mu(n) = \begin{cases}
    1 &\text{if } n = 1,\\
    (-1)^r &\text{if $n = p_1 p_2 \cdots p_r$, where $p_i$ are distinct primes},\\
    0 &\text{if } p^2 \,|\, n ~ \text{ for some prime }p.
\end{cases}
\]
The behavior of the \Mobius{} function is fundamentally linked to the structure of prime numbers and prime factorization.
In broad terms, the \Mobius{} function captures the multiplicative structure of the integer $n$. 
Studying the \Mobius{} function under an additive perspective, i.e. what happens when $n \to \infty$, reveals the interplay between multiplication and addition on the positive integers.

The Mertens function $\mertens(n)$ is defined by taking partial sums of the \Mobius{} function,
\[
    \mertens(n) = \sum_{k = 1}^n \mu(k).
\]
Many properties of this function were studied by Mertens \cite{mertens}.
Significantly, the Riemann hypothesis is equivalent to the asymptotic bound 
\[
    \mertens(n) = O(n^{1/2 + \epsilon}) \qquad\text{for any } \epsilon > 0.
\]
Whether this bounds holds is currently open, and is a subject of active research.
Kotnik and van de Lune
\cite{kotnik-vandelune}
investigate the asymptotics of the Mertens function by numerical experiment.

\subsection{Posets and Hasse diagrams}
\label{sec:partial-order}

The classical definition of the \Mobius{} function depends on the prime factorization of a positive integer, which in turn depends on the relation of integer divisibility.
Integer divisibility defines a {\em partial order} relation on the positive integers.

The classical \Mobius{} function is uniquely characterized by the following equations, coming from the integer divisibility relation.

\begin{enumerate}[(i)]
    \item[(M.1)] For $n = 1$, we have $\mu(1) = 1$.
    
    \item[(M.2)] For any $n \geq 2$, we have $\sum_{d \,|\, n} \mu(d) = 0$.
\end{enumerate}
These relations can be adapted to form the definition of the ``\Mobius{} function'' for an arbitrary partially ordered set.






Before discussing this generalized \Mobius{} function, let us first recall some definitions.
A partially ordered set, also known as a {\em poset}, is a tool for ordering combinatorial objects. 
For more background, see Bona~\cite[Chapter 16]{Bona}. 
A poset $(S, \preccurlyeq)$ consists of an underlying set $S$ and a binary relation ``$\preccurlyeq$" on $S$ satisfying the following axioms: 
\begin{itemize}
    \item (Reflexivity)
    $x \preccurlyeq x$ for all $x \in S$;
    \item (Antisymmetry)
    if $x \preccurlyeq y$ and $y \preccurlyeq x$, then $x = y$;
    \item (Transitivity)
    if $x \preccurlyeq y$ and $y \preccurlyeq z$, then $x \preccurlyeq z$.
\end{itemize}
If these hold, then we call 
$\preccurlyeq$ a {\em partial order relation}.


The {\em Hasse diagram} of a poset is a nice way to represent a poset graphically. 
It is a directed graph whose vertices are the elements of the poset, and whose edges correspond to {covering relations} in the poset.
A {\em covering relation} is a pair $(x,y)$ of distinct elements such that $x \preccurlyeq y$, and there is no element $z \not\in \{x, y\}$ such that $x \preccurlyeq z \preccurlyeq y$.
Rather than indicating the directions of edges with arrows, it is typical to display Hasse diagrams so that all edges implicitly point ``upwards,'' i.e. if $x \preccurlyeq y$ is a covering relation, then we draw the $x$-vertex lower and the $y$-vertex higher in the diagram.

In Figure~\ref{fig:hasse-divisor} is an example of a Hasse diagram of the integers $\{1, 2, \ldots, 20\}$ under the divisibility relation.
With the divisibility relation, the covering relations are those of the form $n$ divides $pn$, where $p$ is a prime number.

\begin{figure}[h]\label{Integer}
\centering
\includegraphics[width= 0.75\textwidth]{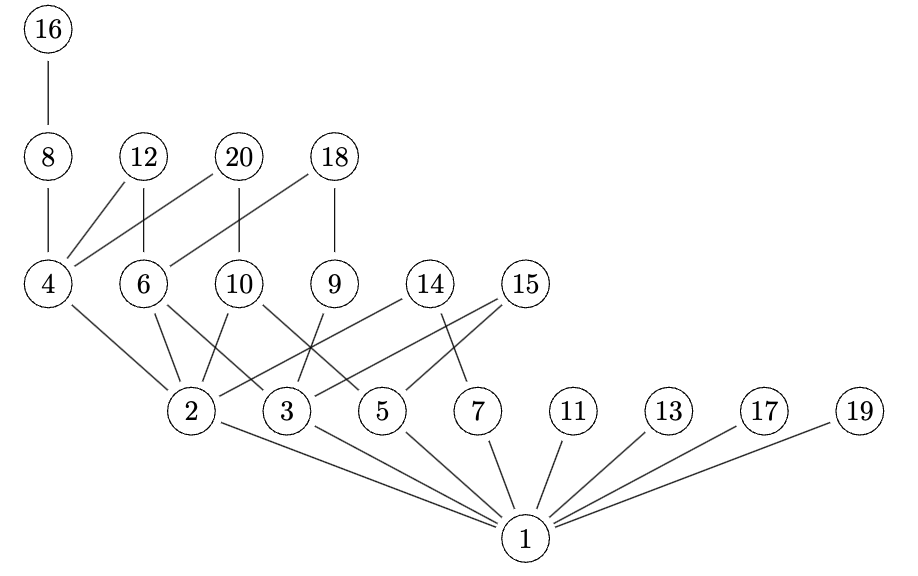}
\caption{Hasse Diagram for Integer Divisibility}
\label{fig:hasse-divisor}
\end{figure}

\subsection{Generalized \Mobius{} functions}\label{sec:gen-mobius}

We now return to the generalized \Mobius{} functions on a poset.
For more details on the \Mobius{} function, see Bona~\cite[Chapter 16.2]{Bona}. 
In order to define the \Mobius{} function for $(S, \preccurlyeq)$ we need one additional technical definition:
we say a poset $(S, \preccurlyeq)$ is {\em locally finite} if for any $x, y \in S$, the set $\{z : x \preccurlyeq z \preccurlyeq y\}$ of elements between $x$ and $y$ is finite.

The {\em \Mobius{} function} of a locally finite poset $(S, \preccurlyeq)$ is defined as follows:
$\mu_S : S \times S \to \ZZ$ satisfies
\begin{enumerate}
  \item[(GM.0)] 
    $\mu_S(x, y) = 0$ if $x \not\preccurlyeq y$;
  \item[(GM.1)] 
    $\mu_S(x, x) = 1$ for all $x \in S$;
  \item[(GM.2)] 
    if $x \preccurlyeq y$ and $x \neq y$, then $\displaystyle \sum_{z :\, x \preccurlyeq z   \preccurlyeq y} \mu_s(x, z) = 0$.
\end{enumerate}


Now suppose there is an element $1 \in S$, such that $1 \preccurlyeq x$ for all $x \in S$.
The \Mobius{} function of $(S, \preccurlyeq)$ is defined by $\mu_S(1, 1) = 1$ and
\[
    \sum_{y \preccurlyeq x} \mu_P(1, y) = 0 \qquad\text{if } x \neq 1.
\]

\subsection{Matrix computation: Zeta matrix and \Mobius{} matrix}\label{sec:mobius-matrix}

We now describe how the generalized \Mobius{} function, defined in the previous section, has an equivalent description in terms of matrices.
This matrix version has the advantage of being straightforward to implement in a programming language, using a well-supported linear algebra library.
For the rest of the paper, we assume for convenience that our poset is defined on the underlying set $S = \NN$.

The {\em zeta matrix} of a poset $P = (\NN, \leqP)$ on the natural numbers is the $\{0,1\}$-valued matrix whose entries are
\[
    Z_{i,j} = \begin{cases}
    1 &\text{if } j \leqP i, \\
    0 &\text{otherwise.}
    \end{cases}
\]

Examples of these matrices using the triangular numbers and integers (for comparison) are shown in Section~\ref{sec:tri-matrices}. 
Creating the zeta matrix for a poset
is equivalent to figuring out the poset relation between all pairs of elements in the poset;
to find a submatrix of the zeta matrix, we just need to figure out the poset relation between pairs of corresponding elements. 
The $i$-th row of the zeta matrix records which elements are larger than $i$ in the poset; 
the $j$-th column of the zeta matrix records which elements are smaller than $j$ in the poset.

We may compute the \Mobius{} function of a poset by finding the matrix inverse of the zeta matrix.
Namely, if $P$ is a locally finite poset, let 
\begin{equation}
    M = Z^{-1}
\end{equation}
where $Z$ denotes the zeta matrix of $P$ from above. 
Then the entries of the matrix $M$ are exactly the \Mobius{} function values:
\begin{equation}\label{eq:mobius-matrix}
    M_{i, j} = \mu_{P} (j, i).
\end{equation}
This is because the properties defining the \Mobius{} function, (GM.0 - GM.2) above, are equivalent to the matrix relation
\[
    [\mu_P(i, j)]_{i, j \in \NN} \, Z = I,
\]
where $I$ denotes the infinite $\NN \times \NN$ identity matrix.

Note: In this paper we focus on studying the first row of the \Mobius{} matrix. 

\section{Poset of triangular numbers}

In this section we define the poset which is the main focus of this paper.
As before, let $\tri(n) = \frac12 n(n + 1)$ denote the $n$-th triangular number.
By abuse of notation, let $\tri$ also denote the set of {\em triangular numbers}, i.e.,
\[
    \tri = \left\{ \textstyle\frac12 n (n + 1) : n = 1, 2, \ldots \right\}
    = \{1, 3, 6, 10, \ldots \}.
\]
Consider the poset $(\NN, \leqtri)$ define by $i \leq j$ if and only if $\tri(i)$ divides $\tri(j)$.

\subsection{Hasse diagram}

Here we show the Hasse diagram of the first $20$ elements of $(\NN, \leqtri)$,
which shows the divisibility relations among the first $20$ triangular numbers,
in Figure~\ref{fig:hasse-triangle}.
Recall that a line is drawn between two numbers if they are ``minimally related'' to one other. 
For example, $\tri(4)$ divides $\tri(19)$, so there is a line from $4$ to $19$. 

\begin{figure}[h!]
\centering
\includegraphics[width= 0.75\textwidth]{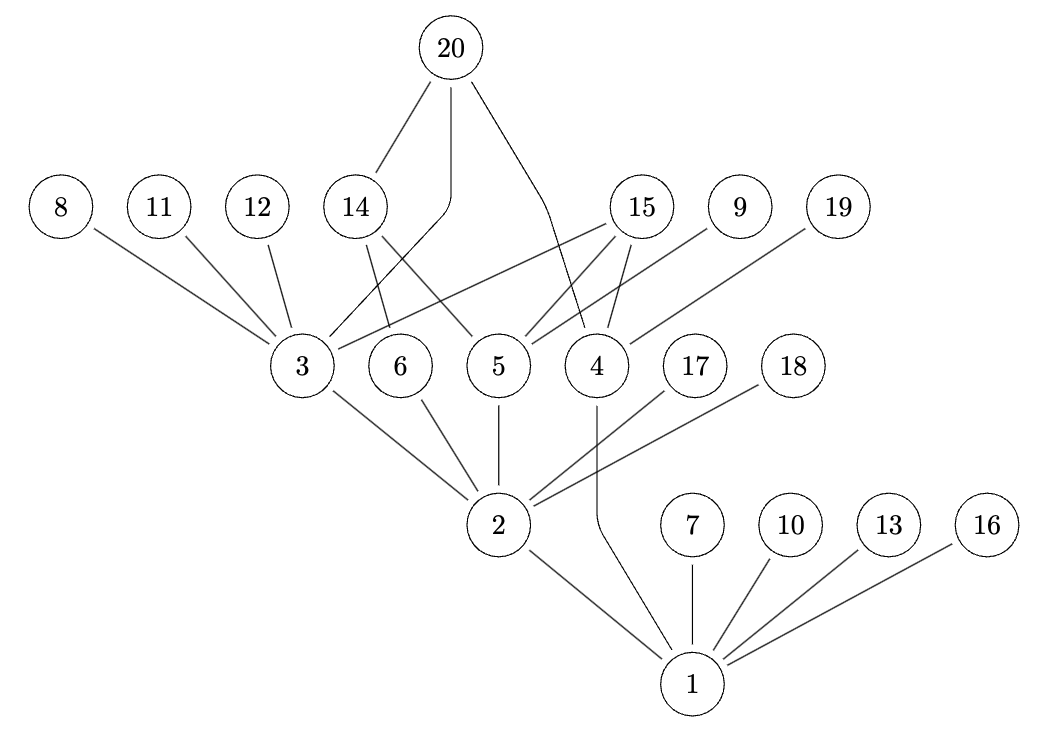}
\caption{Hasse diagram for $\leqtri$, encoding triangular number divisibility.}
\label{fig:hasse-triangle}
\end{figure}

It's important to note that in the Hasse diagram, relations implied by transitivity are not shown.
For example, the relations $5 \leqtri 14$ and $14 \leqtri 20$ are shown by edges in the Hasse diagram, since $\tri(5)$ divides $\tri(14)$ and $\tri(14)$ divides $\tri(20)$. 
But even though $5 \leqtri 20$, we don't have this edge in the Hasse diagram because this relation is already implied by 
the upwards-path of edges from $5$ to $20$.


From the Hasse diagram, we can visibly see that the poset $(\NN, \leqtri)$ is more disordered than the usual divisibility poset, shown earlier in Figure~\ref{fig:hasse-divisor}.
We note some patterns in Section~\ref{sec:triangular-patterns}.

\subsection{Triangular numbers: zeta matrix and \Mobius{} matrix}
\label{sec:tri-matrices}

Let $\mutri$ denote \Mobius{} function for the triangular numbers under the divisibility relation, in the sense defined in Section~\ref{sec:gen-mobius}.
Namely, $\mutri: \NN \to \ZZ$ is the unique function that satisfies
\begin{equation}
    \mutri(1, 1) = 1,
    \qquad 
    \mutri(n) = -\sum_{\substack{d \leqtri n \\[0.2em] d \neq n}} \mutri(1, d) \quad\text{for all } n \geq 2.
\end{equation}

As mentioned earlier in Section~\ref{sec:mobius-matrix}, the values of the \Mobius{} function $\mutri$ can be computed by inverting the zeta matrix.
Each row of the zeta matrix records which elements are greater than a given element, in the poset relation, and each column records which elements are smaller.
For every relation $i \leqtri j$, a $1$ is inserted in the corresponding $i$-th row and $j$-th column of the zeta matrix. 

The following shows the initial part of the zeta matrix:
\begin{equation}\label{eq:zeta-10}
    Z = {\color{gray}\begin{matrix}
        1 \\3 \\6 \\10 \\15 \\ 
        21 \\ 28 \\ 36 \\ 45 \\ 55
    \end{matrix}}
    \begin{bmatrix}
    1 \\
    1 & 1 \\
    1 & 1 & 1 \\
    1 & 0 & 0 & 1 \\
    1 & 1 & 0 & 0 & 1 \\
    1 & 1 & 0 & 0 & 0 & 1 \\
    1 & 0 & 0 & 0 & 0 & 0 & 1 \\
    1 & 1 & 1 & 0 & 0 & 0 & 0 & 1 \\
    1 & 1 & 0 & 0 & 1 & 0 & 0 & 0 & 1 \\
    1 & 0 & 0 & 0 & 0 & 0 & 0 & 0 & 0 & 1
    \end{bmatrix}
\end{equation}

Note that the first column 
is filled with $1$'s for every row. 
A larger zeta matrix for the partial order $(\NN, \leqtri)$, restricted to elements $\{1, \ldots, 20\}$, is shown in Figure~\ref{fig:zeta-20}.


The following shows the initial part of the \Mobius{} matrix, for the poset $(\NN, \leqtri)$:
\begin{equation}
    M = {\color{gray}\begin{matrix}
        1 \\3 \\6 \\10 \\15 \\ 
        21 \\ 28 \\ 36 \\ 45 \\ 55
    \end{matrix}}
    \begin{bmatrix}
    1 \\
    -1 & 1 \\
    0 & -1 & 1 \\
    -1 & 0 & 0 & 1 \\
    0 & -1 & 0 & 0 & 1 \\
    0 & -1 & 0 & 0 & 0 & 1 \\
    -1 & 0 & 0 & 0 & 0 & 0 & 1 \\
    0 & 0 & -1 & 0 & 0 & 0 & 0 & 1 \\
    0 & 0 & 0 & 0 & -1 & 0 & 0 & 0 & 1 \\
    -1 & 0 & 0 & 0 & 0 & 0 & 0 & 0 & 0 & 1
    \end{bmatrix}
\end{equation}
This \Mobius{} matrix $M$ is the inverse of the zeta matrix $Z$ in \eqref{eq:zeta-10}.


\section{Results: data on \Mobius{} values}

In this section, we report some empirical observations concerning the values of the \Mobius{} function $\mutri$.
The values are available on the Online Encyclopedia of Integer Sequences (OEIS) as sequence A350682~\cite{OEIS-mobius}.

\subsection{\Mobius{} values with \texorpdfstring{$m=1$}{m = 1}}

In Figure~\ref{fig:mobius-300}, we show the \Mobius{} values of the partial order $(\NN, \leqtri)$.
The values are highly erratic, rapidly switching between positive and negative values.

\begin{figure}[h!]
    \centering
    \includegraphics[width=4in]{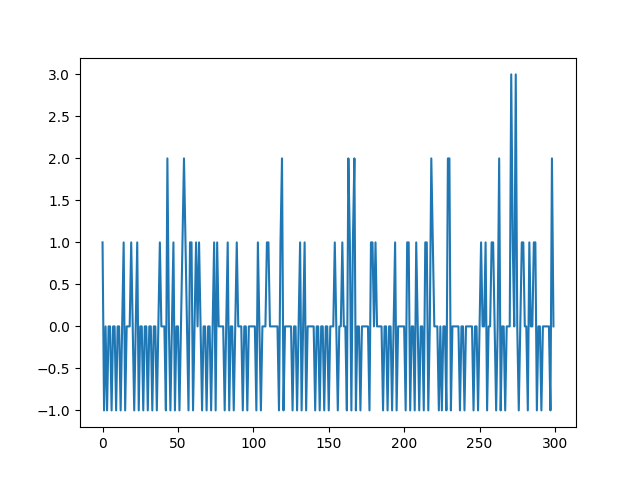}
    \vspace{-0.5cm}
    \caption{The first 300 \Mobius{} values $\mutri(n)$.}
    \label{fig:mobius-300}
\end{figure}

\subsection{Data: Partial sums of \Mobius{} values}

In this section we show data on the partial sums of the \Mobius{} values $\mutri(n)$
\[
    \sum_{i = 1}^n \mutri(i).
\]
This sequence of partial sums is available at the OEIS entry A351167~\cite{OEIS-mobius-sums}.
Figure~\ref{fig:mobius-tri-sums} shows a graph of these partial sums, for up to $n = 10,000$.

\begin{figure}[h]
    \centering
    \includegraphics[scale=0.7]{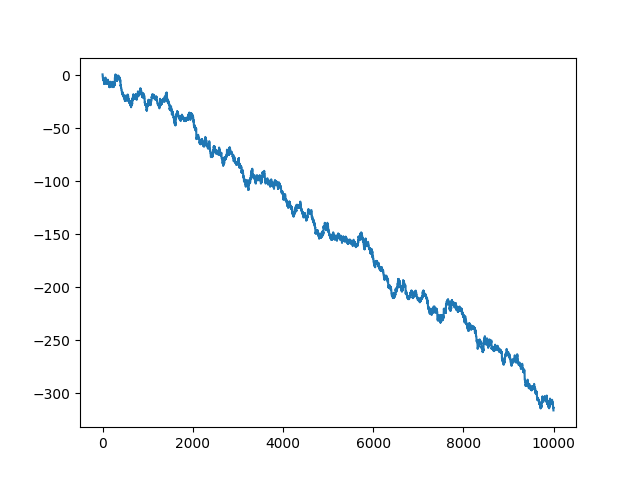}
    \caption{Partial sums of \Mobius{} function values.}
    \label{fig:mobius-tri-sums}
\end{figure}

Unlike Figure~\ref{fig:mobius-300}, Figure~\ref{fig:mobius-tri-sums} shows a clear downward trend.
This leads us to make the following conjecture.

\begin{conj*}[Growth of partial sums of $\mutri$]
There is a positive constant $C$ such that
\begin{equation*}
    \sum_{i=1}^n \mutri(i) \leq - C n \quad\text{for all sufficiently large } n.
\end{equation*}
\end{conj*}

In other words, this conjecture states that the average value of the \Mobius{} function $\mutri$ is eventually bounded above by $-C$, i.e.
\[
    \limsup_{n \to \infty} \frac1{n} \sum_{i = 1}^n \mutri(i) \leq -C .
\]

\subsection{Data: Partial sums of \Mobius{} value absolute values}

In Figure~\ref{fig:abs-mobius-sum} we show the partial sums of the absolute values $|\mutri(n)|$.
In this figure, the trend is even smoother than in Figure~\ref{fig:mobius-tri-sums}.
\begin{figure}[h!]
    \centering
    \includegraphics{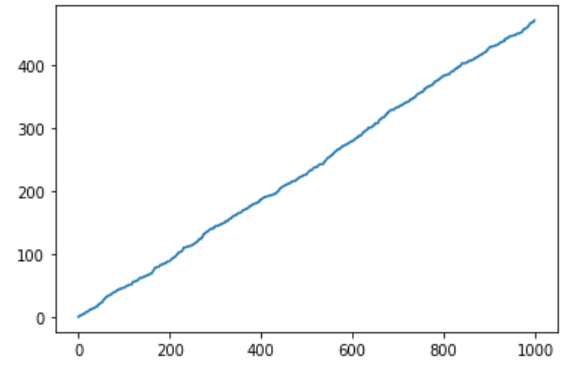}
    \vspace{-0.5cm}
    \caption{Partial sums of \Mobius{} function absolute values $|\mutri(n)|$.}
    \label{fig:abs-mobius-sum}
\end{figure}

This data in Figure~\ref{fig:abs-mobius-sum} leads us to make the following conjecture.
\begin{conj*}[Partial sums of $|\mutri|$]
As $n \to \infty$,
\begin{equation*}
    \sum_{i=1}^n |\mutri(i)| = \frac12 n + o(n).
\end{equation*}
\end{conj*}

\subsubsection{Average magnitude of \Mobius{} values}

The conjecture can be rephrased in terms of the average magnitute of the the \Mobius{} function values.
Namely, that
\[
    \lim_{n \to \infty} \frac{1}{n} \sum_{i = 1}^n |\mutri(i)| = \frac12 .
\]

\subsection{\Mobius{} values with large magnitude}

We observe empirically that the values of the \Mobius{} function $\mutri$ seem to achieve arbitrarily large magnitude.
This is in contrast with the classical \Mobius{} function $\mu$, which only has values in $\{-1, 0, 1\}$.

\begin{table}[h]
    \centering
    \begin{tabular}{||c | c||} 
    \hline
    $M$ & first $n$ such that $|\mutri(n)| \geq M$ \\ [0.5ex] 
    \hline\hline
    1 & 1  \\ 
    2 & 44  \\
    3 & 272 \\
    4 & 1274 \\
    5 & 2639 \\
    6 & 6720 \\
    7 & 3024 \\
    8 & 2079 \\ [1ex] 
    \hline
    \end{tabular}
    \vspace{0.5cm}
    \caption{Inputs of the \Mobius{} function $\mutri$ with increasing magnitude.}
    \label{table:mobius-tri-large}
\end{table}

This data leads us to make the following conjecture.
\begin{conj}
    For any positive integer $M$, there is a positive integer $n$ such that ${|\mutri(n)| \geq M}$.
\end{conj}

\subsection{Two-variable \Mobius{} values}
In Appendix~\ref{sec:mobius-data} 
we show a heatmap illustrating the values of the two-variable \Mobius{} function for $(\NN, \leqtri)$.
This should allow further explorations for patterns in the \Mobius{} values in future work.



\section{Further questions}

\begin{itemize}
    \item 
    If Conjecture~\ref{conj:mobius-sum} holds, what are bounds on the constant $C$?
    Can the exact value of $C$ be computed?
    
    \item 
    Other statistics to measure on $\mutri$?
    
    \item 
    Similar to Conjecture ~\ref{conj:asymptote}, is there an asymptotic relation of the form 
    \begin{equation}
    \lim_{n \to \infty} \sum^n_{i=1} \frac{\mutri (i)}{\tri(i)} \approx D,
    \end{equation}
    where $D$ is some constant?
\end{itemize}

\begin{conj}
\label{conj:asymptote}
There is a positive constant $E$ such that
\begin{equation}
    \lim_{n\to\infty} \sum_{i=1}^n \frac{\mutri(i)}{i} = - E .
\end{equation}
\end{conj}

This constant $-E$, can be seen to be around approximately $-0.239$:

\begin{figure}[h!]
    \centering
    \includegraphics[width=4in]{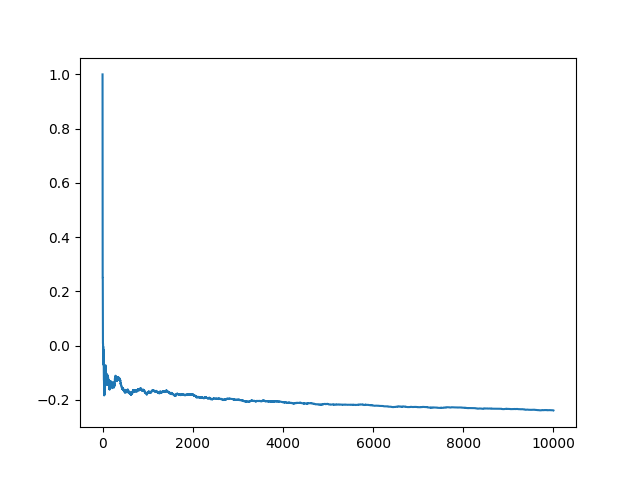}
    \vspace{-0.5cm}
    \caption{Partial sums of $\mutri(n)/n$ up to $10,000$.}
    \label{fig:enter-label}
\end{figure}

\newpage


%
%
\bibliography{triangle-mobius-ref}

\begin{bibdiv}
\begin{biblist}

\bib{Bona}{book}{
      author={B\'{o}na, Mikl\'{o}s},
       title={A walk through combinatorics},
     edition={Fourth},
   publisher={World Scientific Publishing Co. Pte. Ltd., Hackensack, NJ},
        date={2017},
        ISBN={978-981-3148-84-0},
        note={An introduction to enumeration and graph theory, With a foreword
  by Richard Stanley},
      review={\MR{3560666}},
}

\bib{Burton}{book}{
      author={Burton, David~M.},
       title={Elementary number theory},
     edition={Second},
   publisher={W. C. Brown Publishers, Dubuque, IA},
        date={1989},
        ISBN={0-697-05919-7},
      review={\MR{990017}},
}

\bib{matplotlib}{article}{
      author={Hunter, J.~D.},
       title={Matplotlib: A 2d graphics environment},
        date={2007},
     journal={Computing in Science \& Engineering},
      volume={9},
      number={3},
       pages={90\ndash 95},
}

\bib{kotnik-vandelune}{article}{
      author={Kotnik, Tadej},
      author={van~de Lune, Jan},
       title={On the order of the {Mertens} function},
    language={English},
        date={2004},
        ISSN={1058-6458},
     journal={Exp. Math.},
      volume={13},
      number={4},
       pages={473\ndash 481},
}

\bib{mertens}{article}{
      author={Mertens, Franz},
       title={{\"U}ber eine zahlentheoretische {F}unktion},
        date={1897},
     journal={Sitzungsberichte Akad. Wiss. Wien IIa},
      volume={106},
       pages={761\ndash 830},
}

\bib{mobius}{article}{
      author={M\"{o}bius, A.~F.},
       title={\"{U}ber eine besondere {A}rt von {U}mkehrung der {R}eihen},
        date={1832},
     journal={Journal f\"{u}r die reine und angewandte Mathematik},
      volume={9},
       pages={105\ndash 123},
}

\bib{OEIS-mobius}{misc}{
      author={{OEIS Foundation Inc. (2022)}},
       title={Entry {A}350682 in the {O}n-{L}ine {E}ncyclopedia of {I}nteger
  {S}equences},
        note={\url{http://oeis.org/A350682}},
}

\bib{OEIS-mobius-sums}{misc}{
      author={{OEIS Foundation Inc. (2022)}},
       title={Entry {A}351167 in the {O}n-{L}ine {E}ncyclopedia of {I}nteger
  {S}equences},
        note={\url{http://oeis.org/A351167}},
}

\bib{plotly}{misc}{
      author={{Plotly Technologies Inc.}},
       title={Collaborative data science},
   publisher={Plotly Technologies Inc.},
     address={Montreal, QC},
        date={2015},
         url={https://plot.ly},
        note={\url{https://plot.ly}},
}

\bib{python}{book}{
      author={Van~Rossum, Guido},
      author={Drake, Fred~L.},
       title={Python 3 reference manual},
   publisher={CreateSpace},
     address={Scotts Valley, CA},
        date={2009},
        ISBN={1441412697},
}

\end{biblist}
\end{bibdiv}
\bibliographystyle{abbrv}

\newpage
%
%
\appendix
\section{\Mobius{} matrices}\label{sec:mobius-data}

In this appendix, we visualize the values of the two-variable \Mobius{} function of $(\NN, \leqtri)$.
A heatmap showing the values of the \Mobius{} matrix is given in Figure~\ref{fig:mobius-heatmap-tri}.
Positive values are indicated by blue and negative values are indicated by red; zero values are light gray.



\begin{figure}[h!]
    \centering
    \includegraphics[width= 6in]{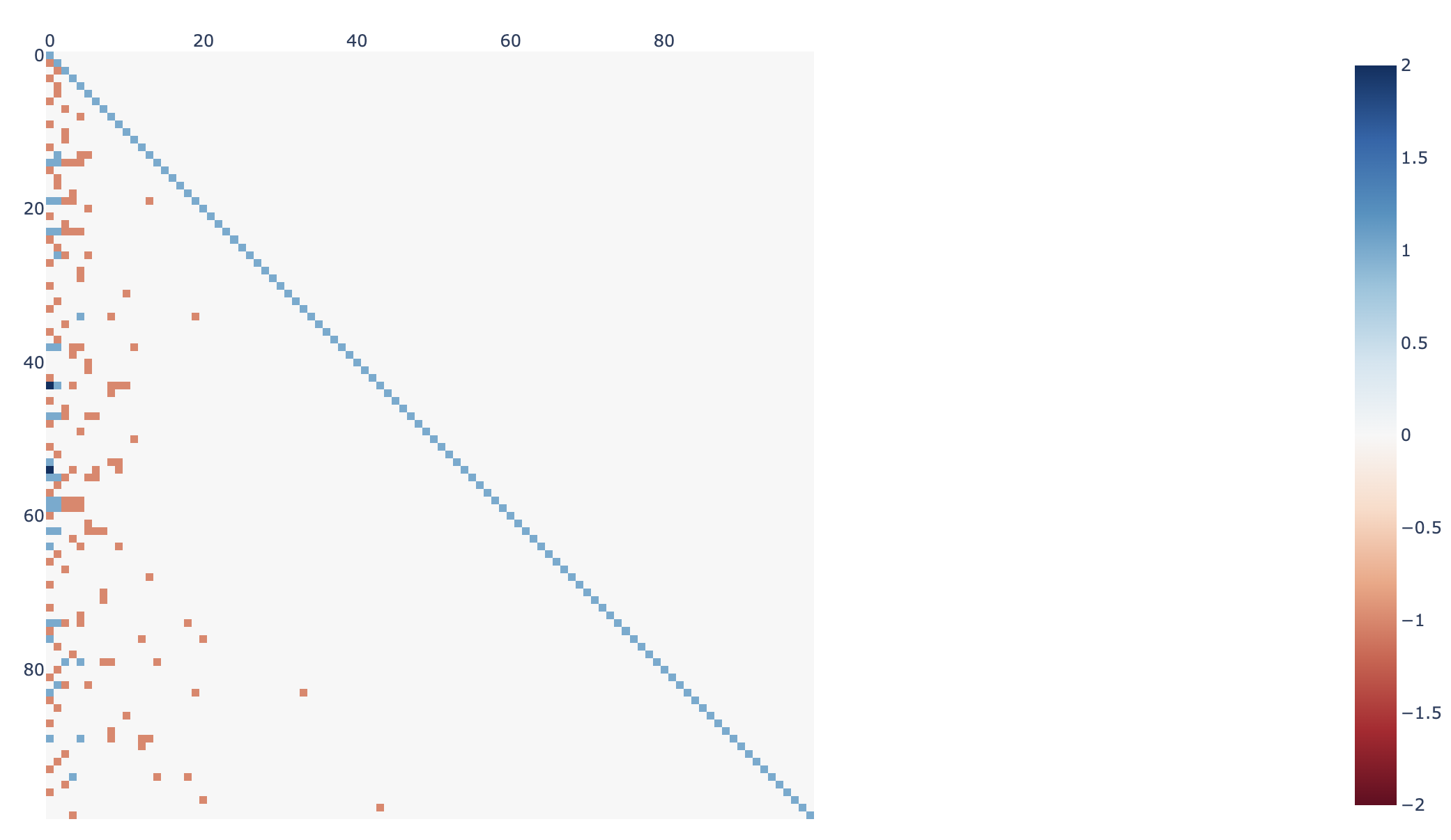}
    \caption{Two-variable \Mobius{} function for $(\NN, \leqtri)$ for $1 \leq m, n \leq 100$.}
    \label{fig:mobius-heatmap-tri}
\end{figure}

\begin{figure}[h!]
    \centering
    \includegraphics[width = 3in]{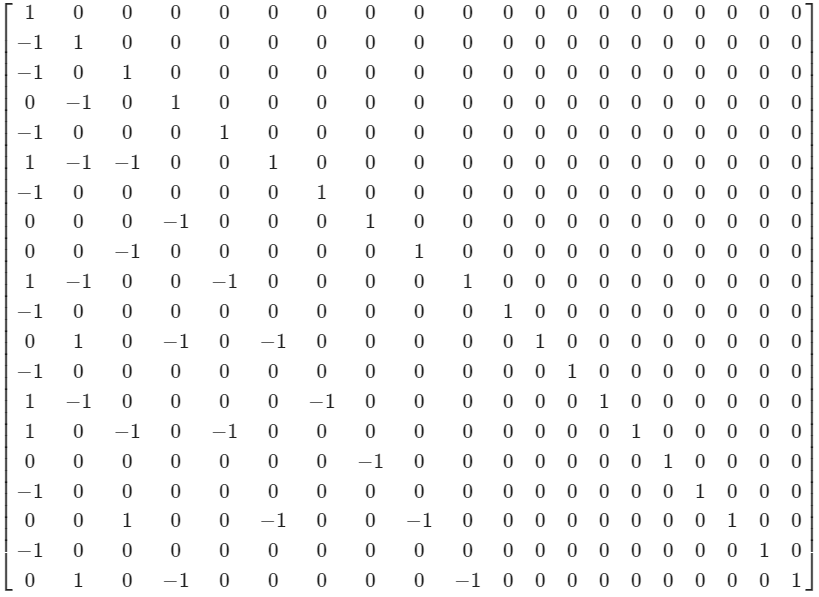}
    \caption{\Mobius{} matrix of 20 Integers}
    \label{fig:mobius-matrix-classical-20}
\end{figure}

For comparison, Figure~\ref{fig:mobius-heatmap-classical} shows the values of the classical \Mobius{} function.
\begin{figure}[h!]
    \centering
    \includegraphics[width = 6in]{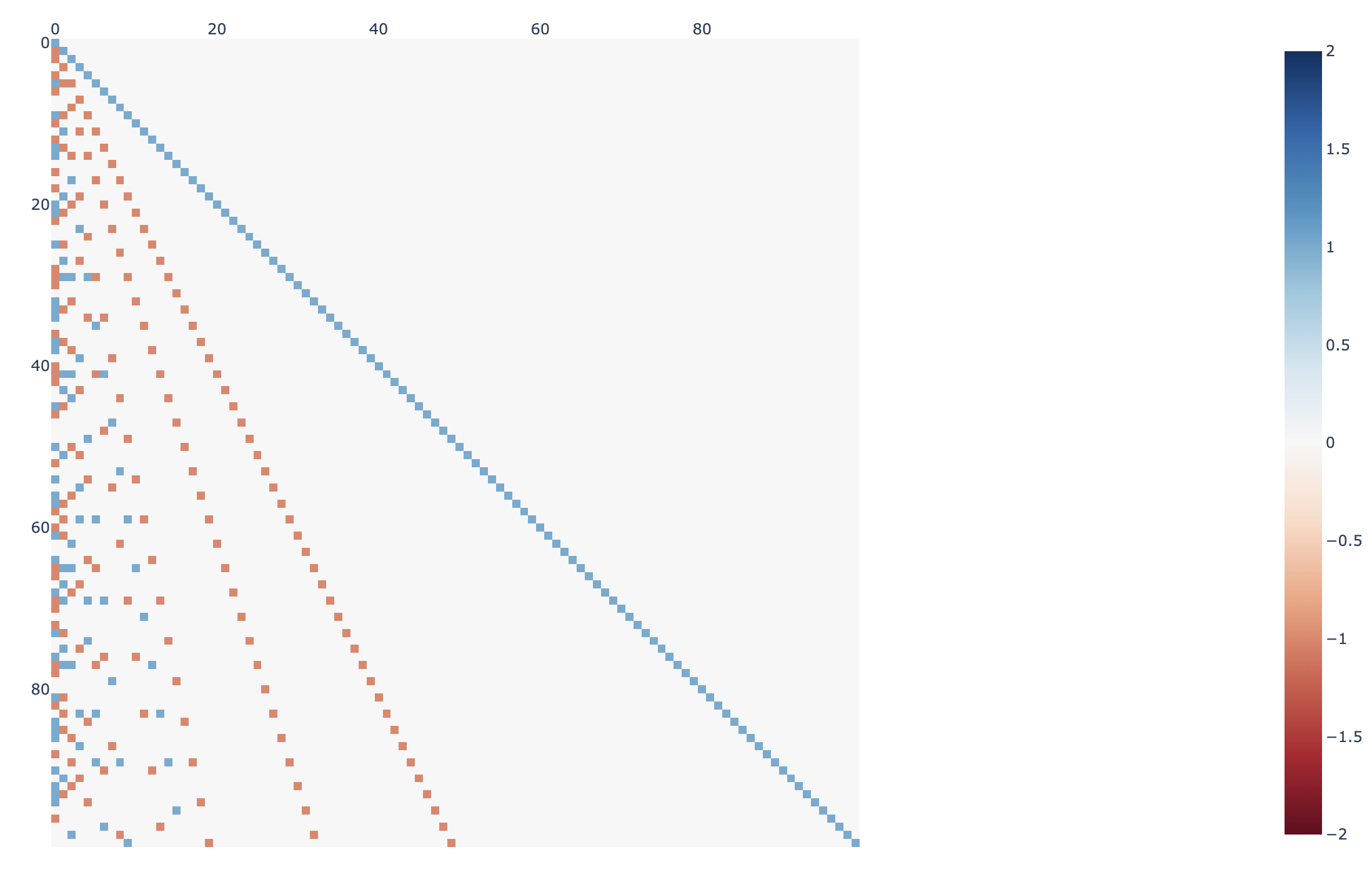}
    \caption{Classical two-variable \Mobius{} function for $1 \leq m, n \leq 100$.}
    \label{fig:mobius-heatmap-classical}
\end{figure}

\begin{figure}[h!]
    \centering
    \includegraphics[width = 3in]{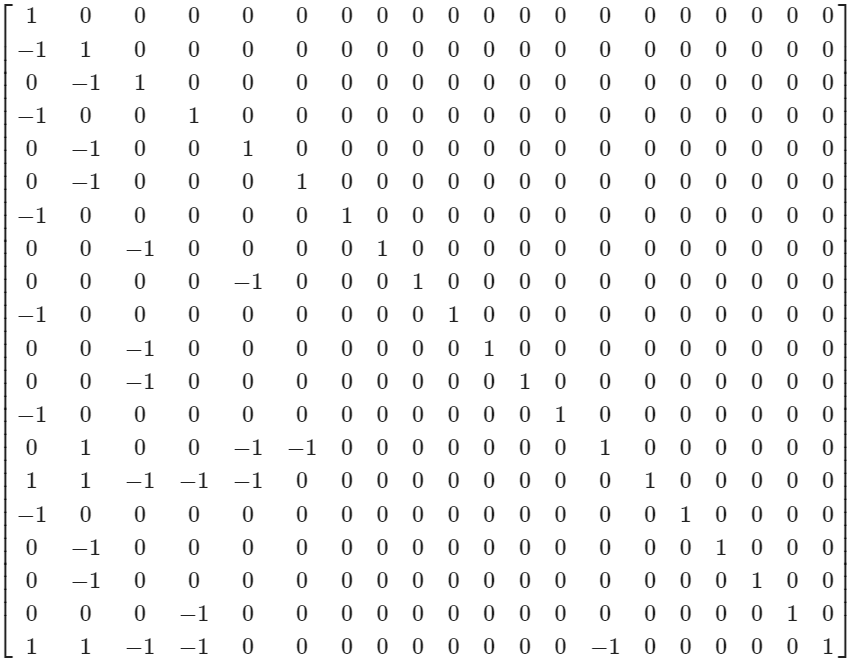}
    \caption{\Mobius{} matrix of 20 Triangular Numbers}
    \label{fig:mobius-matrix-20}
\end{figure}

\begin{figure}[h!]
    \centering
    \includegraphics[scale=0.8]{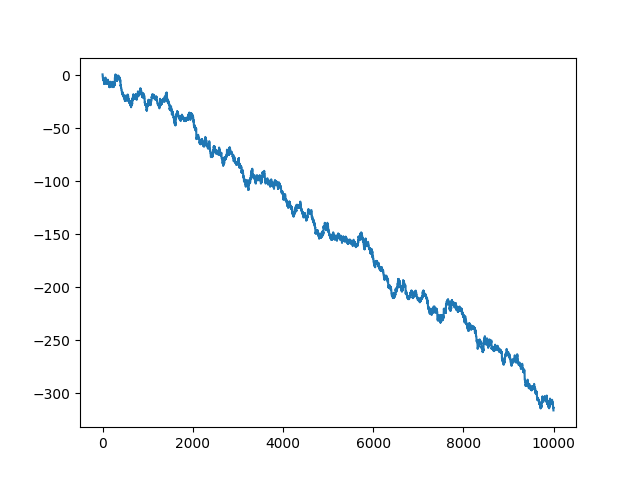}
    \caption{Sums of \Mobius{} values $\mutri(n)$ up to 10,000.}
    \label{fig:mobius-sum-10000}
\end{figure}

\newpage
\section{Code}

In this section, we include the 

The first is for generating the zeta matrix.
Note: It was slightly faster and easier to generate matrices by dividing column by the rows, so at the end we transposed the matrix to get the desired result. 
\begin{mdframed}[
backgroundcolor=light-gray, 
linecolor=light-gray,
roundcorner=20pt,
leftmargin=10, 
rightmargin=1, 
innerleftmargin=15, 
innertopmargin=5,
innerbottommargin=5, 
outerlinewidth=1
]
\begin{lstlisting}
from sympy import *
import matplotlib.pyplot as plt
import numpy as np

def triangular_numbers(n):
    # returns a list of triangular numbers
    return [(x * (x + 1) // 2) for x in range(1,n+1)]
    
def zeta_matrix(n):
    # returns zeta matrix for partial order on first n triangular numbers
    lst = triangular_numbers(n)
    zeta_array = [[0 if a % b != 0 else 1 for a in lst] for b in lst]
    Z = Matrix(zeta_array).T
    return Z

# to create zeta matrix for first 100 triangular numbers:
zeta_matrix(100)
\end{lstlisting}
\end{mdframed}

\bigskip
The second is for generating the \Mobius{} matrix.
\begin{mdframed}[backgroundcolor=light-gray, roundcorner=10pt,leftmargin=1, rightmargin=1, innerleftmargin=15, innertopmargin=10,innerbottommargin=10, outerlinewidth=1, linecolor=light-gray]
\begin{lstlisting}
def mobius_matrix(a):
    Z = Zeta_Matrix(a)
    M = Z ** -1
    return M

def plot_mobius_values(n):
    M = mobius_matrix(n)
    a = M[0, :n].tolist() 
    plt.plot(a)
    plt.legend()
    plt.show()
\end{lstlisting}
\end{mdframed}

The third is for generating the sum of the mobius values:
\begin{mdframed}[backgroundcolor=light-gray, roundcorner=10pt,leftmargin=1, rightmargin=1,innerleftmargin=15,innertopmargin=15,innerbottommargin=15, outerlinewidth=1,linecolor=light-gray]
\begin{lstlisting}
number = int(input("Number: "))

M = mobius_matrix(triangular_numbers(number))
N = M[0, :].tolist()

def sum_function(lst):
    sum_list = [sum(lst[:i+1]) for i in range(len(lst))]
    return sum_list

S = sum_function(N[0])
plt.plot(S)
plt.show()
\end{lstlisting}
\end{mdframed}

\bigskip
The fourth is for generating the absolute mobius value sums.
\begin{mdframed}[backgroundcolor=light-gray, roundcorner=10pt,leftmargin=1, rightmargin=1,innerleftmargin=15, innertopmargin=15,innerbottommargin=15, outerlinewidth=1,linecolor=light-gray]
\begin{lstlisting}
number = int(input("Number: "))

M = mobius_matrix(triangular_numbers(number))
N = M[0, :].tolist()

S = sum_function(abs_value(N[0]))
plt.plot(S)
plt.show()
#slope approaching 0.5
\end{lstlisting}
\end{mdframed}

The slope was found using the equation
\begin{equation}
    \frac{y_2 - y_1}{x_2 - x_1}
\end{equation}
where the change in $y$ was the difference between the last M\"{o}bius value and the first one, and the change in $x$ was just the number input minus $0$.

The fifth is for creating heatmaps for visualizing the \Mobius{} matrix, which are generated using the Plotly package~\cite{plotly}.

\begin{mdframed}[backgroundcolor=light-gray, roundcorner=10pt,leftmargin=1, rightmargin=1, innerleftmargin=15, innertopmargin=15,innerbottommargin=15, outerlinewidth=1, linecolor=light-gray]
\begin{lstlisting}
import plotly.express as px
from mobius_matrix import mobius_matrix, triangular_numbers

def plot_mobius_values():
    M = mobius_matrix(triangular_numbers())
    M = M.transpose()
    fig = px.imshow(M, color_continuous_scale='RdBu', color_continuous_midpoint=0.0)
    fig.update_layout(
    xaxis={'side': 'top'})
    fig.show()

plot_mobius_values()
\end{lstlisting}
\end{mdframed}

The sixth is for the partial sums:

\begin{mdframed}[backgroundcolor=light-gray, roundcorner=10pt,leftmargin=1, rightmargin=1, innerleftmargin=15, innertopmargin=15,innerbottommargin=15, outerlinewidth=1, linecolor=light-gray]
    \begin{lstlisting}
        number = int(input("Number: "))

M = mobius_matrix(triangular_numbers(number))
N = M[0, :].tolist()

def partial_sums(lst):
    result = [sum(lst[:i+1]) for i in range(len(lst))]
    return result

n_divided = [N[i]/(i+1) for i in range(len(N))]

S = partial_sums(n_divided)

print(min(S))


plt.plot(S)
plt.show()
    \end{lstlisting}
\end{mdframed}

\section{Triangular number divisibility patterns}
\label{sec:triangular-patterns}

In this section, we investigate some properties of when one triangular number divides another.
The results of this section are independent of the rest of the paper.

\label{sec:zeta-patterns}
Let 
$\tri(n) = \frac12 n(n+1)$
denote the $n$-th triangular number.
Recall that the partial order $(\NN, \leqtri)$ records, for each pair of positive integers $i,j$, whether or not $\tri(i) \,|\, \tri(j)$ holds.
The following statements may be useful for studying asymptotics of \Mobius{} values $\mutri(m,n)$ as $n\to \infty$, for fixed $m\geq 2$.

\begin{prop}
For any $n$, $\tri(n)$ divides $\tri(n(n+1))$.
\end{prop}

\begin{proof}
We can calculate directly
\[
    \tri({n(n+1)}) = \frac{n(n+1)(n(n+1)+1)}{2} = \tri(n) (n(n+1) + 1).
\]
Therefore the ratio $\tri(n(n+1)) / \tri(n)$ simplifies to
\[
    \frac{\tri(n(n+1))}{\tri(n)} = n(n+1) + 1 
\]
which is an integer for any $n$.
\end{proof}

\begin{prop}
\label{prop:2}
For any $n$, $\tri(n)$ divides $\tri( \frac12 n(n+1) )$ if and only if 
$n \equiv 1$ or $2$ mod 4.
\end{prop}

\begin{proof}
We have
\[
    \tri\left( \frac{n(n+1)}{2}\right) = \frac{1}{2} \left(\frac{n(n+1)}{2}\right) \cdot \left(\frac{n(n+1)}{2} + 1\right) 
    = \tri(n) \cdot \frac{n(n+1) + 2}{4}.
\]
So, for the proposition it suffices to prove that  $\frac{1}{4}(n(n+1) + 2)$ is an integer exactly when $n \equiv 1$ or $2$ modulo $ 4$.
So, we have to find all $n$ such that $n(n+1) + 2 = n^2 + n + 2$ is congruent to 0 modulo 4, i.e.
\begin{equation}
    \label{eq:1} n(n+1) \equiv 2 \quad(\mathrm{mod}\quad 4)
\end{equation}

So, in the set of residues modulo $4$, what values of $n$ will satisfy equation \eqref{eq:1}? 

\begin{align*}
   n = 0 &\qquad\Rightarrow\qquad  n(n+1) = 0 \not\equiv 2 \quad(\mathrm{mod}\quad 4) \\
    n = 1 &\qquad\Rightarrow\qquad n(n+1) = 2 \equiv 2 \quad(\mathrm{mod}\quad 4) \\
    n = 2 &\qquad\Rightarrow\qquad n(n+1) = 6 \equiv 2 \quad(\mathrm{mod}\quad 4) \\
    n = 3 &\qquad\Rightarrow\qquad n(n+1) = 12 \not\equiv 2 \quad(\mathrm{mod}\quad 4)
\end{align*}
Therefore, $n(n+1) \equiv 2 \mod 4$ if and only if $n \equiv 1$ or $2 \mod 4$,
as desired.
%
\end{proof}

\end{document}